\documentclass[twoside,sort&compress]{article}

\usepackage{hyperref}
\usepackage{fancyhdr}
\usepackage{amsmath}
\usepackage{amssymb}
\usepackage{geometry}
\usepackage{cite}
\usepackage{graphicx}

\newtheorem{theorem}{Theorem}[section]

\newtheorem{lemma}[theorem]{Lemma}

\newenvironment{proof}[1][Proof]{\noindent\textbf{#1.} }{\ \rule{0.5em}{0.5em}}
\numberwithin{equation}{section}
\begin{document}

\setcounter{page}{1} \thispagestyle{plain}

%{\small Vol.**, No.**, **-**}

\begin{center}
{\Large NORM ESTIMATES FOR BESSEL-RIESZ OPERATORS ON GENERALIZED MORREY SPACES}
\bigskip

{\large Mochammad Idris$^{1}$, Hendra Gunawan$^2$, and Eridani$^3$}

\bigskip

$^{1}$Department of Mathematics, Institut Teknologi Bandung,\\[0pt]
Bandung 40132, Indonesia\\[0pt]
[{\it Permanent Address}:
Department of Mathematics, Lambung Mangkurat University,\\[0pt]
Banjarbaru Campus, Banjarbaru 70714, Indonesia]\\[6pt]

$^{2}$Department of Mathematics, Institut Teknologi Bandung,\\[0pt]
Bandung 40132, Indonesia\\[6pt]

$^{3}$Department of Mathematics, Airlangga University,\\[0pt]
Campus C Mulyorejo, Surabaya 60115, Indonesia\\[6pt]

E-mail addresses: $^{1}$idemath@gmail.com, $^{2}$hgunawan@math.itb.ac.id,
and $^{3}$eridani.dinadewi@gmail.com

\bigskip
\end{center}

\begin{quote}
\textbf{Abstract.} We revisit the properties of Bessel-Riesz operators and
present a different proof of the boundedness of these operators on generalized
Morrey spaces. We also obtain an estimate for the norm of these operators on
generalized Morrey spaces in terms of the norm of their kernels on an associated
Morrey space. As a consequence of our results, we reprove the boundedness of
fractional integral operators on generalized Morrey spaces, especially of
exponent 1, and obtain a new estimate for their norm.



\end{quote}

%% Main matter

\section{Introduction}

Integral operators such as maximal operators and fractional integral operators
have been studied extensively in the last four decades. Here we are interested
in Bessel-Riesz operators, which are related to fractional integral operators.
Let $0<\alpha<n$ and $\gamma \ge 0$. The operator $I_{\alpha,\gamma}$ which maps
every $f\in L^{p}_{\text loc}(\mathbb{R}^{n}),\ 1\leq p<\infty$, to
$$%
I_{\alpha ,\gamma }f\left( x\right) :=\int_{\mathbb{R}^{n}}K_{\alpha,\gamma}(x-y)
f(y) dy=K_{\alpha,\gamma}*f(x),\quad x\in \mathbb{R}^n,
$$
where $K_{\alpha ,\gamma }(x) :=\frac{|x|^{\alpha-n}}{(1+|x|)^{\gamma}}$, is
called \textit{Bessel-Riesz operator}, and the kernel $K_{\alpha,\gamma}$ is
called \textit{Bessel-Riesz kernel}. The boundedness of these operators on Morrey
spaces and on generalized Morrey spaces was studied in \cite{Idris1, Idris2}.

Let $1\leq p<\infty$ and $\phi:\mathbb{R}^{+}\rightarrow\mathbb{R}^{+}$ be of class
$\mathcal{G}_p$, that is $\phi$ is almost decreasing [$\exists\,C>0$ such that
$\phi(r)\ge C\phi(s)$ for $r\le s$] and $\phi^p(r)r^n$ is almost increasing [$\exists
\,C>0$ such that $\phi^p(r)r^n\le C\phi^p(s)s^n$ for $r\le s$]. Clearly if $\phi$
is of class $\mathcal{G}_p$, then $\phi$ satisfies the {\it doubling condition},
that is, there exists $C>0$ such that $\frac{1}{C}\le \frac{\phi(r)}{\phi(s)} \le C$
whenever $1\le \frac{r}{s}\le 2$. We define the {\it generalized Morrey space}
$L^{p,\phi}(\mathbb{R}^n)$ to be the set of all functions $f\in L^{p}_{\rm loc}
(\mathbb{R}^{n})$ for which
$$
\|f\|_{L^{p,\phi}}:=\sup_{B=B(a,r)} \frac{1}{\phi(r)}\left(\frac{1}{|B|}
\int_{B} |f(x)|^{p} dx\right)^{\frac{1}{p}}<\infty,
$$
where $|B|$ denotes the Lebesgue measure of $B$. (Recall that the Lebesgue measure
of $B=B(a,r)$ is $|B(a,r)|=C_n r^n$ for every  $a \in \mathbb{R}^n$ and $r>0$,
where $C_n>0$ depends only on $n$.)

If $1 \leq p\leq q<\infty$ and $\phi(r):=C_n r^{-\frac{n}{q}}\ (r>0)$, then
$L^{p,\phi}(\mathbb{R}^{n})$ is the classical Morrey space $L^{p,q}(\mathbb{R}^{n})$,
which is equipped by
$$
\|f\|_{L^{p,q}}:=\sup_{B=B(a,r)} |B|^{\frac{1}{q}-\frac{1}{p}}\left(\int_{B}
|f(x)|^{p}dx\right)^{\frac{1}{p}}.
$$
Particularly, for $p=q$, $L^{p,p}(\mathbb{R}^{n})$ is the Lebesgue space
$L^{p}(\mathbb{R}^{n})$.

In \cite{Idris1}, we know that for $\gamma>0$, $K_{\alpha,\gamma}$ is a member of
$L^t(\mathbb{R}^n)$ spaces for some values of $t$ depending on $\alpha$ and $\gamma$.
It follows from Young's inequality \cite{Grafakos} that
$$
\|I_{\alpha,\gamma}f\|_{L^{q}}\leq \| K_{\alpha,\gamma}\|_{L^{t}}
\|f\|_{L^{p}},\quad f\in L^{p}(\mathbb{R}^{n}),
$$
whenever $1\leq p<t^{\prime},\ \frac{1}{q}=\frac{1}{p}-\frac{1}{t'}$ (where $t'$
denotes the dual exponent of $t$) and $\frac{n}{n+\gamma-\alpha}<t<\frac{n}{n-\alpha}$.
This tells us that $I_{\alpha,\gamma}$ is bounded from $L^p(\mathbb{R}^n)$ to
$L^q(\mathbb{R}^n)$ with $\|I_{\alpha,\gamma}\|_{L^p\to L^q}\le\|K_{\alpha,\gamma}\|_{L^t}$.
In \cite{Idris2}, it is also shown that $I_{\alpha,\gamma}$ is bounded on
generalized Morrey spaces but without a good estimate for its norm as on Morrey
spaces. We shall now refine the results, by estimating the norms of the operators
more carefully through the membership of $K_\alpha$ in Morrey spaces.

Note that for $\gamma=0,$ $I_{\alpha,0}=I_\alpha$ is the {\it fractional integral operator}
with kernel $K_\alpha(x):=|x|^{\alpha-n}$. Hardy-Littlewood \cite{Hardy1, Hardy2} and
Sobolev \cite{Sobolev} proved the boundedness of $I_{\alpha}$ on Lebesgue spaces.
The boundedness of $I_{\alpha}$ on Morrey spaces is proved by Spanne \cite{Peetre},
and improved by Adams \cite{Adams} and Chiarenza-Frasca \cite{Chiarenza}.
Later, Nakai \cite{Nakai94} obtained the boundedness of $I_{\alpha}$ on generalized
Morrey spaces, which can be viewed as an extension of Spanne's result. In 2009,
Gunawan-Eridani \cite{Gunawan} proved the boundedness of $I_{\alpha }$ on generalized
Morrey spaces which extends Adams' and Chiarenza-Frasca's results.

In this paper, we give a new proof of the boundedness of $I_{\alpha,\gamma}$ on
generalized Morrey spaces. At the same time, an upper bound for the norm of the
operators is obtained. As a consequence of our result, we have an estimate for
the norm of $I_\alpha$ (from a generalized Morrey space to another) in terms of
the norm of $K_\alpha$ on the associated Morrey space. A lower bound for the norm
of the operators is discussed in \S3.

\section{The Boundedness of $I_{\alpha,\gamma}$ on Generalized Morrey Spaces}

We begin with a lemma about the membership of $K_{\alpha}$ in some Morrey spaces.
Note that throughout this paper, the letters $C$ and $C_k$ denote constants which
may change from line to line.

\begin{lemma}\label{BB04:L03f}
If $0<\alpha<n$, then $K_{\alpha} \in L^{s,t}(\mathbb{R}^{n})$
where $1\le s<t=\frac{n}{n-\alpha}$.
\end{lemma}

\begin{proof}
Let $0<\alpha<n$. Take an arbitrary $B=B(a,R)$ where $a\in \mathbb{R}^{n}$ and $R>0$.
For $1\le s<t=\frac{n}{n-\alpha}$, we observe that
$$
|B|^{\frac{s}{t}-1} \int_{B} K_{\alpha}^{s}(x)dx \le |B(0,R)|^{\frac{s}{t}-1} \int_{B(0,R)}
|x|^{(\alpha-n)s}dx \le C\,R^{n(\frac{s}{t}-1)} R^{n(1-\frac{s}{t})}= C.
$$
By taking the supremum over $B=B(a,R)$, we obtain $\|K_{\alpha}\|_{L^{s,t}}^s \le C$.
Hence $K_{\alpha} \in L^{s,t}(\mathbb{R}^{n})$.
\end{proof}

\medskip

\noindent{\tt Remark}. For $0<\alpha<n$ and $\gamma>0$, we know that $K_{\alpha,\gamma} \in
L^{t}(\mathbb{R}^{n})$ for $\frac{n}{n+\gamma-\alpha}<t<\frac{n}{n-\alpha}$ \cite{Idris1}.
By the inclusion property of Morrey spaces (see \cite{GHLM}), we have
$
K_{\alpha,\gamma}\in L^{t}(\mathbb{R}^{n})=L^{t,t}(\mathbb{R}^{n})\subseteq
L^{s,t}(\mathbb{R}^{n}),
$
for $1\leq s \leq t$ and $\frac{n}{n+\gamma-\alpha}<t<\frac{n}{n-\alpha}$. Moreover, because
$K_{\alpha,\gamma}(x)\leq K_{\alpha}(x)$ for every $x \in \mathbb{R}^n$, $K_{\alpha,\gamma}$
is also contained in $L^{s,t}(\mathbb{R}^{n})$ for $1\leq s<t=\frac{n}{n-\alpha}$.

\medskip

As a counterpart of the results in \cite{Idris2, Idris1}, we have the following theorem on
the boundedness of $I_{\alpha,\gamma}$ on Morrey spaces. Note particularly that the estimate
holds for $p_1=1$.

\begin{theorem}\label{BB03:T04}
If $0<\alpha <n$ and $\gamma\ge 0$, then $I_{\alpha,\gamma}$ is bounded from
$L^{p_1,q_1}$ to $L^{p_2,q_2}$ with
$$
\|I_{\alpha,\gamma}f\|_{L^{p_{2},q_{2}}}\leq C\,\|K_{\alpha,\gamma}\|_{L^{s,t}}
\|f\|_{L^{p_{1},q_{1}}},\quad f\in L^{p_{1},q_{1}}(\mathbb{R}^{n}),
$$
whenever $1\leq p_{1}\le q_{1}<\frac{n}{\alpha},\ \frac{1}{p_{2}}=\frac{1}{p_{1}}-\frac{1}{s'}$,
and $\frac{1}{q_{2}}=\frac{1}{q_{1}}-\frac{1}{t'}$, with $1\le s < t=\frac{n}{n-\alpha}$
(for $\gamma\ge 0$) or $1\le s\le t$ and $\frac{n}{n+\gamma-\alpha}<t<\frac{n}{n-\alpha}$
(for $\gamma>0$).
\end{theorem}

Theorem \ref{BB03:T04} is in fact a special case of the boundedness of $I_{\alpha,\gamma}$
on generalized Morrey spaces, which is stated in the following theorem.

\begin{theorem}\label{BB04:L04c}
Let $0<\alpha<n$ and $\gamma\ge 0$. If $\phi:\mathbb{R}^{+}\rightarrow\mathbb{R}^{+}$ is of
class $\mathcal{G}_{p_1}$ such that $\int_R^\infty \phi(r)r^{\frac{n}{t'}-1} dr\le C\,\phi(R)
R^{\frac{n}{t'}}$ for every $R>0$, then $I_{\alpha,\gamma}$ is bounded from $L^{p_1,\phi}$ to
$L^{p_2,\psi}$ where $\psi(r):=\phi(r)r^{\frac{n}{t'}}$, with
$$
\|I_{\alpha,\gamma}f\|_{L^{p_2,\psi}} \le C\,\|K_{\alpha,\gamma}\|_{L^{s,t}}
\|f\|_{L^{p_1,\phi}},\quad f \in L^{p_1,\phi}(\mathbb{R}^{n}),
$$
whenever $1\le p_1 <\frac{n}{\alpha}$ and $\frac{1}{p_2}=\frac{1}{p_1}-\frac{1}{s'}$,
with $1\le s < t=\frac{n}{n-\alpha}$ (for $\gamma\ge 0$) or $1\le s\le t$ and
$\frac{n}{n+\gamma-\alpha}<t<\frac{n}{n-\alpha}$ (for $\gamma>0$).
\end{theorem}

\begin{proof}
Suppose that $\gamma>0$ and all the hypotheses hold. For $f \in L^{p_1,\phi}(\mathbb{R}^{n})$
and $B=B(a,R)$ where $a\in\mathbb{R}^n$ and $R>0$, write
$$
f := f_{1}+f_{2} := f_{\chi_{\widetilde{B}}} + f_{\chi_{\widetilde{B}^{c}}},
$$
where $\widetilde{B}=B(a,2R)$ and $\widetilde{B}^{c}$ denotes its complement.
To estimate $I_{\alpha,\gamma}f_{1}$, we observe that for every $x \in B$,
H\"{o}lder's inequality gives
\begin{align*}
|I_{\alpha,\gamma}f_{1}(x)|
\leq & \int_{\widetilde{B}} K_{\alpha,\gamma}(x-y) |f(y)|dy\\
= & \int_{\widetilde{B}} K_{\alpha,\gamma}^{\frac{s}{p_2}}(x-y)
|f(y)|^{\frac{p_1}{p_2}} K_{\alpha,\gamma}^{\frac{p_2-s}{p_2}}(x-y)|f(y)|^{\frac{p_2-p_1}{p_2}} dy\\
\leq & \left( \int_{\widetilde{B}} K_{\alpha,\gamma}^{s}(x-y)|f(y)|^{p_1} dy \right)^{\frac{1}{p_2}}
\left( \int_{\widetilde{B}} K_{\alpha,\gamma }^{\frac{p_2-s}{p_2-1}}(x-y)
|f(y)|^{\frac{p_2-p_1}{p_2-1}}dy \right)^{\frac{1}{p'_2}}.
\end{align*}
Meanwhile, we have
$$
\int_{\widetilde{B}} K_{\alpha,\gamma}^{\frac{p_2-s}{p_2-1}}(x-y)|f(y)|^{\frac{p_2-p_1}{p_2-1}} dy
\leq \left(\int_{\widetilde{B}} K_{\alpha,\gamma}^{s}(x-y) dy \right)^{p'_2(\frac{1}{s}-\frac{1}{p_2})}
\left(\int_{\widetilde{B}} |f(y)|^{p_1} dy \right)^{\frac{p'_2}{s'}}.
$$
Therefore we obtain
\begin{align*}
|I_{\alpha,\gamma}f_{1}(x)|
\le & \left( \int_{\widetilde{B}} K_{\alpha,\gamma}^{s}(x-y)|f(y)|^{p_1}dy \right)^{\frac{1}{p_2}}
\left( \int_{\widetilde{B}} K_{\alpha,\gamma}^{s}(x-y) dy \right)^{\frac{1}{s}-\frac{1}{p_2}}
\left( \int_{\widetilde{B}} |f(y)|^{p_1} dy \right)^{\frac{1}{s'}}\\
\le & \left( \int_{\widetilde{B}} K_{\alpha,\gamma}^{s}(x-y)|f(y)|^{p_1} dy \right)^{\frac{1}{p_2}}
\times C\,R^{n(1-\frac{s}{t})(\frac{1}{s}-\frac{1}{p_2}) + \frac{n}{s'}} \phi^{\frac{p_1}{s'}}(2R)
\|K_{\alpha,\gamma}\|_{L^{s,t}}^{1-\frac{s}{p_2}} \|f\|_{L^{p_1,\phi}}^{\frac{p_1}{s'}}.
\end{align*}
We take the $p_2$-th power and integrate both sides over $B$ to get
\begin{align*}
\int_{B}|I_{\alpha,\gamma}f_{1}(x)|^{p_2} dx
\leq & \int_{B} \int_{\widetilde{B}} K_{\alpha,\gamma}^{s}(x-y)|f(y)|^{p_1} dy\,dx \\
&\ \times \left(C\,R^{n(1-\frac{s}{t})(\frac{1}{s}-\frac{1}{p_2})+\frac{n}{s'}} \phi^{\frac{p_1}{s'}}(2R)
\|K_{\alpha,\gamma}\|_{L^{s,t}}^{1-\frac{s}{p_2}} \|f\|_{L^{p_1,\phi}}^{\frac{p_1}{s'}} \right)^{p_2}.
\end{align*}
By Fubini's theorem, we have
\begin{align*}
\int_{B}| I_{\alpha,\gamma}f_{1}(x)|^{p_2} dx
\le & \int_{\widetilde{B}} |f(y)|^{p_1} \left(\int_{B} K_{\alpha,\gamma}^{s}(x-y) dx\right)dy \\
&\ \times \left(C\,R^{n(1-\frac{s}{t})(\frac{1}{s}-\frac{1}{p_2})+\frac{n}{s'}} \phi^{\frac{p_1}{s'}}(2R)
\|K_{\alpha,\gamma}\|_{L^{s,t}}^{1-\frac{s}{p_2}}\|f\|_{L^{p_1,\phi}}^{\frac{p_1}{s'}} \right)^{p_2}\\
\le &\ C\,R^{n(1-\frac{s}{t})} \|K_{\alpha,\gamma}\|_{L^{s,t}}^s \int_{\widetilde{B}} |f(y)|^{p_1} dy \\
&\ \times \left(R^{n(1-\frac{s}{t})(\frac{1}{s}-\frac{1}{p_2})+\frac{n}{s'}} \phi^{\frac{p_1}{s'}}(2R)
\|K_{\alpha,\gamma}\|_{L^{s,t}}^{1-\frac{s}{p_2}}\|f\|_{L^{p_1,\phi}}^{\frac{p_1}{s'}} \right)^{p_2}\\
\le &\ C\,R^{n(1-\frac{s}{t})+n}\phi^{p_1}(2R) \|K_{\alpha,\gamma}\|_{L^{s,t}}^s \|f\|_{L^{p_1,\phi}}^{p_1}\\
&\ \times \left(R^{n(1-\frac{s}{t})(\frac{1}{s}-\frac{1}{p_2})+\frac{n}{s'}} \phi^{\frac{p_1}{s'}}(2R)
\|K_{\alpha,\gamma}\|_{L^{s,t}}^{1-\frac{s}{p_2}} \|f\|_{L^{p_1,\phi}}^{\frac{p_1}{s'}} \right)^{p_2}\\
\le &\ C\,|B|\,\psi^{p_2}(R) \|K_{\alpha,\gamma}\|_{L^{s,t}}^{p_2} \|f\|_{L^{p_1,\phi}}^{p_2},
\end{align*}
whence
$$
\frac{1}{\psi(R)} \left(\frac{1}{|B|} \int_{B} |I_{\alpha,\gamma}f_{1}(x)|^{p_2} dx
\right)^{\frac{1}{p_2}}\leq C\,\|K_{\alpha,\gamma}\|_{L^{s,t}} \|f\|_{L^{p_1,\phi}}.
$$

Next, we estimate $I_{\alpha,\gamma}f_{2}$. For every $x\in B=B(a,R)$, we observe that
\begin{align*}
| I_{\alpha,\gamma}f_{2}(x) |
\leq & \int_{\widetilde{B}^{c}} K_{\alpha,\gamma}(x-y)|f(y)| dy \\
\leq & \int_{|x-y |\ge R} K_{\alpha,\gamma}(x-y)|f(y)| dy \\
= & \sum_{k=0}^{\infty} \int_{2^{k}R\leq |x-y|<2^{k+1}R}K_{\alpha,\gamma}(x-y)|f(y)| dy \\
\leq &\ \sum_{k=0}^{\infty} K_{\alpha,\gamma}(2^k R) \int_{2^{k}R\le |x-y|<2^{k+1}R} |f(y)| dy\\
\leq &\ C \sum_{k=0}^{\infty} K_{\alpha,\gamma}(2^k R) (2^k R)^{\frac{n}{p'_1}}
\left(\int_{2^{k}R\leq |x-y| <2^{k+1}R} |f(y)|^{p_{1}}dy\right)^{\frac{1}{p_{1}}}\\
\leq &\ C\,\|f\|_{L^{p_1,\phi}} \sum_{k=0}^{\infty} K_{\alpha,\gamma}(2^kR)(2^kR)^{n}\phi(2^kR).
\end{align*}
For every $k\in \mathbb{Z}$, we have
$$
K_{\alpha,\gamma}(2^k R)\le C\,{(2^k R)^{-\frac{n}{s}}}\left(\int_{2^{k}R\le |x-y|<2^{k+1}R}
K^s_{\alpha,\gamma}(x-y)dy \right)^{\frac{1}{s}}\le C\,(2^k R)^{-\frac{n}{t}}
\|K_{\alpha,\gamma}\|_{L^{s,t}}.
$$
Since $\int_R^\infty \phi(r)r^{\frac{n}{t'}-1}dr \le C\,\phi(R)R^{\frac{n}{t'}}$, we get
\begin{align*}
| I_{\alpha,\gamma}f_{2}(x) |
\leq &\ C\,\|K_{\alpha,\gamma}\|_{L^{s,t}} \|f\|_{L^{p_1,\phi}} \sum_{k=0}^{\infty}
(2^k R)^{\frac{n}{t'}}\phi(2^k R)\\
\leq &\ C\,\|K_{\alpha,\gamma}\|_{L^{s,t}} \|f\|_{L^{p_1,\phi}} \int_R^\infty
\phi(r)r^{\frac{n}{t'}-1}dr\\
\leq &\ C\,\|K_{\alpha,\gamma}\|_{L^{s,t}} \|f\|_{L^{p_1,\phi}} \phi(R) R^{\frac{n}{t'}}\\
= &\ C\,\|K_{\alpha,\gamma}\|_{L^{s,t}} \|f\|_{L^{p_1,\phi}} \psi(R).
\end{align*}
Raising to the $p_2$-th power and integrating over $B$, we obtain
$$
\int_{B}| I_{\alpha,\gamma}f_{2}(x) |^{p_{2}}dx \leq C\,\bigl(\|K_{\alpha,\gamma}\|_{L^{s,t}}
\|f\|_{L^{p_{1},\phi}}\bigr)^{p_2} \psi^{p_2}(R)|B|,
$$
whence
$$
\frac{1}{\psi(R)}\left(\frac{1}{|B|}\int_{B}| I_{\alpha,\gamma}f_{2}(x) |^{p_{2}}dx
\right)^{\frac{1}{p_2}}\leq C\,\|K_{\alpha,\gamma }\|_{L^{s,t}}\|f\|_{L^{p_{1},\phi}}.
$$

Combining the two estimates for $I_{\alpha,\gamma}f_1$ and $I_{\alpha,\gamma}f_2$, we obtain
$$
\frac{1}{\psi(R)}\left(\frac{1}{|B|}
\int_{B}| I_{\alpha,\gamma}f(x) |^{p_{2}}dx\right)^{\frac{1}{p_2}}
\leq C\,\|K_{\alpha,\gamma}\|_{L^{s,t}}\|f\|_{L^{p_{1},\phi}}.
$$
Since this inequality holds for every $a\in\mathbb{R}^n$ and $R>0$, it follows that
$$
\|I_{\alpha,\gamma}f\|_{L^{p_{2},\psi}}\leq C\,\| K_{\alpha,\gamma}\|_{L^{s,t}}
\|f\|_{L^{p_{1},\phi}},
$$
as desired.

We may repeat the same argument and use Lemma \ref{BB04:L03f} to obtain the same inequality
for the case where $\gamma=0$ and $1\le s < t=\frac{n}{n-\alpha}$.
\end{proof}

\medskip

\noindent{\tt Remark}. Theorems \ref{BB03:T04} and \ref{BB04:L04c} give us upper estimates
for the norm of the Bessel-Riesz operators (from one Morrey space to another). In particular,
for $\gamma=0$, we have an estimate for the norm of the fractional integral operator
$I_\alpha$ in terms of the norm of its kernel (on the associated Morrey space), which follows
from the inequality
$$
\|I_{\alpha}f\|_{L^{p_{2},\psi}}\leq C\,\|K_{\alpha}\|_{L^{s,t}} \|f\|_{L^{p_{1},\phi}},
$$
for $1\le p_1 <\frac{n}{\alpha}$ and $\frac{1}{p_2}=\frac{1}{p_1}-\frac{1}{s'}$, with
$1\le s< t=\frac{n}{n-\alpha}$.

\medskip

In the following section, we discuss lower estimates for the norm of the operators
in terms of the norm of the Bessel-Riesz kernel (on some Morrey spaces).

\section{An Estimate for the Norm of the Operators}

Recall that if $(X, \| \cdot \|_{X})$ and $(Y, \| \cdot \|_{Y})$ are normed
spaces and that $T : (X,\| \cdot \|_{X}) \rightarrow (Y,\| \cdot \|_{Y})$ is
a linear operator, then the norm of $T$ (from $X$ to $Y$) is defined by
$$
\|T\|_{X\to Y} :=\sup_{f \neq 0} \frac{\|Tf\|_{Y}}{\|f\|_{X}}.
$$
Knowing that the Bessel-Riesz operator $I_{\alpha,\gamma}$ is a linear operator on
Morrey spaces, we would like to estimate the norm of $I_{\alpha,\gamma}$ from a
(generalized) Morrey space to another. We obtain the following result.

\begin{theorem}\label{gbr:04b}
Let $0<\alpha<n$, $\gamma\ge0$, and $\phi$ is of class $\mathcal{G}_{p_1}$ where $1\le p_1
<\frac{n}{\alpha}$. If $\phi(r)r^n$ is almost increasing and for every $R>0$ we have
(i) $\int_R^\infty \phi(r)r^{\frac{n}{t'}-1}\,dr\le C_1\phi(R)R^{\frac{n}{t'}}$,
(ii) $\int_0^R \phi^{p_1}(r) r^{n-1}\,dr\leq C_2\,\phi^{p_1}(R)R^n$, and
(iii) $\int_0^R \frac{r^{n-1}}{\phi^{s'}(r)r^{ns'}}\,dr \le \frac{C_3 R^n}{\phi^{s'}(R)R^{ns'}}$,
where $1\le p_1<t$ and $1< s < t=\frac{n}{n-\alpha}$ (for $\gamma\ge 0$) or $1\le p_1 \le t$,
$1< s\le t$, and $\frac{n}{n+\gamma-\alpha}<t<\frac{n}{n-\alpha}$ (for $\gamma>0$), then we have
$$
C_4 \| K_{\alpha,\gamma} \|_{L^{p_{1},t}} \leq \| I_{\alpha,\gamma}\|_{L^{p_1,\phi}\to
L^{p_2,\psi}} \le C_5 \| K_{\alpha,\gamma} \|_{L^{s,t}},
$$
whenever $\frac{1}{p_2}=\frac{1}{p_1}-\frac{1}{s'}$ and $\psi(r):=\phi(r)r^{\frac{n}{t'}}$.
In particular, for $\gamma=0$, $1\le p_1<t$, and $1< s < t=\frac{n}{n-\alpha}$, we have
$$
C_4 \| K_{\alpha} \|_{L^{p_{1},t}} \leq \| I_{\alpha}\|_{L^{p_1,\phi}\to L^{p_2,\psi}}
\le C_5 \| K_{\alpha} \|_{L^{s,t}},
$$
whenever $\frac{1}{p_2}=\frac{1}{p_1}-\frac{1}{s'}$ and $\psi(r):=\phi(r)r^{\frac{n}{t'}}$.
\end{theorem}

\begin{proof}
Suppose that $\gamma>0$ and all the hypotheses hold. By Theorem \ref{BB04:L04c}, we already
have
$$
\|I_{\alpha,\gamma}\|_{L^{p_1,\phi}\to L^{p_2,\psi}}\le C\,\| K_{\alpha,\gamma} \|_{L^{s,t}}.
$$
To prove the lower estimate, put $\rho(r):=\phi(r)r^n$. Let $B=B(a,R)$ where
$a\in\mathbb{R}^n$ and $R>0$. By our assumptions on $\phi$, we have
$$
|B|^{\frac{1}{t}}\psi(R)\left(\frac{1}{|B|}\int_B \rho^{-s'}(|x|)\,dx\right)^{\frac{1}{s'}}
\le C\,\phi(R)R^{\frac{n}{s}}\left(\int_0^R \frac{r^{n-1}}{\phi^{s'}(r)r^{ns'}}\,dr
\right)^{\frac{1}{s'}}\le C.
$$
Now take $f_0(x):=\phi(|x|)$. Here $\|f_0\|_{L^{p_1,\phi}}\approx 1$.
Moreover, one may compute that
$$
I_{\alpha,\gamma}f_0(x)\ge \int_{B(x,2|x|)} K_{\alpha,\gamma}(x-y)f_0(y)dy
\ge C\,K_{\alpha,\gamma}(2x)\phi(|x|)|x|^n=C\,\rho(|x|)K_{\alpha,\gamma}(x),
$$
for every $x\in\mathbb{R}^n$. It follows that
$$
\|\rho(|\cdot|)K_{\alpha,\gamma}\|_{L^{p_2,\psi}}\le C\,\|I_{\alpha,\gamma}f_0\|_{L^{p_2,\psi}}
\le C\,\|I_{\alpha,\gamma}\|_{L^{p_1,\phi}\to L^{p_2,\psi}}.
$$
Next, by H\"older's inequality, we have
$$
\left(\int_B K_{\alpha,\gamma}^{p_1}(x)dx\right)^{\frac{1}{p_1}} \le
\left(\int_B \rho^{-s'}(|x|)\,dx\right)^{\frac{1}{s'}}
\left(\int_B \bigl[\rho(|x|) K_{\alpha,\gamma}(x)\bigr]^{p_2}dx\right)^{\frac{1}{p_2}},
$$
whence
\begin{align*}
|B|^{\frac{1}{t}-\frac{1}{p_1}}\left(\int_B K_{\alpha,\gamma}^{p_1}(x)dx\right)^{\frac{1}{p_1}}\le
&\ |B|^{\frac{1}{t}}\psi(R)\left(\frac{1}{|B|}\int_B \rho^{-s'}(|x|)\,dx\right)^{\frac{1}{s'}}\\
&\ \times \frac{1}{\psi(R)}\left(\frac{1}{|B|}\int_B \bigl[\rho(|x|)K_{\alpha,\gamma}(x)
\bigr]^{p_2}dx\right)^{\frac{1}{p_2}}\\
\le &\ C\,\|I_{\alpha,\gamma}\|_{L^{p_1,\phi}\to L^{p_2,\psi}}.
\end{align*}
By taking the supremum over $B=B(a,R)$, we conclude that
$$
C\,\| K_{\alpha,\gamma} \|_{L^{p_1,t}} \leq \|I_{\alpha,\gamma}\|_{L^{p_1,\phi}\to L^{p_2,\psi}},
$$
as desired.

The same argument applies for the case where $\gamma=0$, with $1\le p_1<t$ and $1< s < t=\frac{n}
{n-\alpha}$.
\end{proof}

\medskip

\noindent{\tt Remark}. One may observe that the constants $C_4$ and $C_5$ in Theorem \ref{gbr:04b}
depend on $\phi, n, p_1, s$, and $t$, but not on $\alpha$ and $\gamma$. Although the lower
and the upper bound are not comparable, we may still get useful information from these estimates,
especially for the norm of the operator $I_{\alpha}$ from $L^{p_1,\phi}$ to $L^{p_2,\psi}$.
Observe that for $1\le p_1< t=\frac{n}{n-\alpha}$, we have $\|K_{\alpha}\|_{L^{p_1,t}}^{p_1}=
\frac{C}{(\alpha-n)p_1+n}\ge \frac{C}{\alpha}.$ Hence, if all the hypotheses in Theorem 3.1 hold
for the case where $\gamma=0$, then we obtain $\|I_{\alpha}\|_{L^{p_1,\phi} \to L^{p_2,\psi}}
\ge \frac{C}{\alpha}$, which blows up when $\alpha\to 0^+$. For $\phi(r):=r^{-\frac{n}{q_1}}$
with $1\le p_1<q_1<\min \{s,\frac{n}{\alpha}\}$ and $1<s<\frac{n}{n-\alpha}$, our result reduces 
to the estimate $\|I_{\alpha}\|_{L^{p_1,q_1}\to L^{p_2,q_2}} \ge \frac{C}{\alpha}$ where 
$\frac{1}{p_2}=\frac{1}{p_1}-\frac{1}{s'}$ and $\frac{1}{q_2}=\frac{1}{q_1}-\frac{\alpha}{n}$. 
A similar behavior of the norm of $I_\alpha$ from $L^{p_1}$ to $L^{p_2}$ for $\frac{1}{p_2}=
\frac{1}{p_1}-\frac{\alpha}{n}$ when $\alpha\to 0^+$ is observed in \cite[Chapter 4]{Lieb}.

\bigskip

\textbf{Acknowledgements.} The first and second authors are supported by ITB Research \&
Innovation Program 2016. All authors would like to thank the anonymous referee for his/her
careful reading and useful comments on the earlier version of this paper.

\end{document}